\documentclass{amsart}
\usepackage{amscd,amsmath,amssymb,amsthm,latexsym,pinlabel}

\def\co{\colon\thinspace}
\newcommand{\C}{{\mathbb C}}
\newcommand{\D }{{\mathbb D}}
\newcommand{\E }{{\mathbb E}}
\newcommand{\N }{{\mathbb N}}
\newcommand{\R}{{\mathbb R}}
\newcommand{\Hy}{{\mathbb H}}

\newcommand{\bi}{{\bf i}}
\newcommand{\ip}{\,\rule{2.3mm}{.2mm}\rule{.2mm}{2.3mm}\; }

\newtheorem{thm}{Theorem}
\newtheorem{lem}[thm]{Lemma}
\newtheorem{prop}[thm]{Proposition}

\theoremstyle{definition}
\newtheorem*{defn}{Definition}
\newtheorem*{rem}{Remark}
\newtheorem*{ack}{Acknowledgements}

\begin{document}

\title[Gibbons--Hawking ansatz and Blaschke products]{A homogeneous
Gibbons--Hawking ansatz and Blaschke products}

\author{Hansj\"org Geiges}
\address{Mathematisches Institut, Universit\"at zu K\"oln,
Weyertal 86--90, 50931 K\"oln, Germany}
\email{geiges@math.uni-koeln.de}
\thanks{H.~G. is partially supported by
         DFG grant GE 1245/1-2 within the framework of the
         Schwer\-punkt\-programm ``Globale Differentialgeometrie''.}

\author[Jes\'us Gonzalo]{Jes\'us Gonzalo P\'erez}
\address{Departamento de Matem\'aticas,
Universidad Aut\'onoma de Ma\-drid, 28049 Madrid, Spain.}
\email{jesus.gonzalo@uam.es}
\thanks{J.~G. is partially supported by
         grants MTM2004-04794 and MTM2007-61982 from MEC Spain.}

\date{}

\begin{abstract}
A homogeneous Gibbons--Hawking ansatz is described, leading to
$4$-dimensional hyperk\"ahler metrics with homotheties. In combination
with Blaschke products on the unit disc in the complex plane, this ansatz
allows one to construct infinite-dimensional families of such
hyperk\"ahler metrics that are, in a suitable sense, complete.
Our construction also gives rise to incomplete metrics on $3$-dimensional
contact manifolds that induce complete Carnot--Carath\'eodory
distances.
\end{abstract}

\maketitle

\section{Introduction}
The aim of this paper is to present an intriguing construction
of hyperk\"ahler structures.

In \cite[Theorem 10]{gego08} we exhibited a global rigidity of
$4$-dimensional hyperk\"ahler metrics: if such a metric $g$ admits
a homothetic vector field with a compact transversal, then $g$ is
flat. We proved this by an argument involving the integral formula for the
signature of a compact $4$-manifold, applied to a quotient of a
neighbourhood of the transversal. That line of reasoning obviously
suggests the question whether the compactness hypothesis can be
weakened to a completeness condition.

In the present paper we define the natural notion of completeness
for a Riemannian metric with a homothety (slice-completeness) and
give a construction leading to non-flat, slice-complete
hyperk\"ahler structures. In Section~\ref{section:GH} we discuss
a homogeneous Gibbons--Hawking ansatz, which forms the
basis for our construction. In Section~\ref{section:canonical} we
derive explicit formul\ae\ for various metrics that arise
in this construction. The main part of the construction of our
examples is contained in Section~\ref{section:example}. In view of a related
{\em incompleteness} result from~\cite{gego08}, see
Theorem~\ref{thm:incomplete} below, it was to be expected (and is
confirmed here) that
the construction of such slice-complete examples would be quite delicate.
So it is all the more surprising that our construction actually yields
an infinite-dimensional family of isometry classes of such
structures (Section~\ref{section:many}). In Section~\ref{section:non-flat}
it is shown that our construction does indeed give rise to
{\em non-flat} hyperk\"ahler metrics. In Section~\ref{section:CC}
we relate our construction to the theory of taut contact
spheres developed in~\cite{gego08}; this relation originally motivated
the search for the hyperk\"ahler metrics described here. In that context we
describe another surprising phenomenon, namely, examples of
incomplete Riemannian metrics giving rise to complete
Carnot--Carath\'eodory distances.
\section{A homogeneous Gibbons--Hawking ansatz}
\label{section:GH}
The Gibbons--Hawking ansatz~\cite{giha78} allows one to construct
hyperk\"ahler metrics with an $S^1$-invariance.
We want to study metrics arising from this ansatz, subject to an
additional homogeneity property amounting to the existence of a
homothetic vector field.

\begin{defn}
A vector field $Y$ is {\bf homothetic} for the Riemannian metric
$g$ if it satisfies $L_Yg=g$. The {\bf canonical slice}
corresponding to such a vector field is the subset defined by the
equation $g(Y,Y)=1$.
\end{defn}

The canonical slice is a hypersurface transverse to $Y$ and, if
the flow of $Y$ is complete, it intersects each orbit exactly
once. For a cone metric $g=e^{2s}(ds^2+\overline{g})$ and
$Y=\partial_s$, the canonical slice is the hypersurface $\{ s=0\}$
orthogonal to $Y$. Most homothetic fields, however, are not
orthogonal to any hypersurface; in such situations our definition
still gives a natural choice of transversal.

A Riemannian metric on a product $M\times\R$ with translation
along the $\R$-factor as homotheties is necessarily incomplete
in the $\R$-direction: proper paths of the form
$\{ p\}\times (-\infty ,s_0]\subset M\times\R$
have finite length in such a metric. Therefore, the best one
can aim for is completeness in the transverse directions.

\begin{defn}
A Riemannian metric on a product $M\times\R$ with translation
along the $\R$-factor as homotheties is {\bf slice-complete} if
the canonical slice is complete in the induced metric.
\end{defn}

For the construction of our examples, we shall be working on a
$4$-manifold $W$ of the form $W=\Sigma \times\R_t\times S^1_\theta$
with $\Sigma$ an open surface; the subscripts denote the respective
coordinates. We look for hyperk\"ahler structures
$(g,\Omega_1,\Omega_2,\Omega_3)$ with the following properties:
\begin{itemize}
\item[(i)] The flow of $\partial_\theta$ preserves the metric $g$, and
$\partial_{\theta}$ is
a Hamiltonian vector field for each of the symplectic
forms~$\Omega_i$.
\item[(ii)] The vector field $\partial_t$ satisfies
$L_{\partial_t}\Omega_i=\Omega_i$, $i=1,2,3$, hence also
$L_{\partial_t}g=g$. Notice that this is stronger than just being
homothetic.
\end{itemize}

The partial differential equations for a hyperk\"ahler structure
linearise under condition~(i) to the $3$-dimensional Laplace equation.
Under the additional condition (ii), one can reduce these equations
further to the Cauchy--Riemann equations in real
dimension~2. We next expand on these two claims.

Given a $3$-manifold $M$, any hyperk\"ahler structure on the product
$M\times S^1_\theta$ satisfying condition~(i) can be described by
the Gibbons--Hawking ansatz. In this ansatz, one selects Hamiltonian
functions $x_1,x_2,x_3$ such that $dx_i=\partial_\theta\ip\Omega_i$.
Then there exist a unique $1$-form $\eta$ and a unique positive
function $V$ giving the following expressions for the symplectic
forms $(\Omega_1,\Omega_2,\Omega_3)$, where $(i,j,k)$ runs over
the cyclic permutations of $(1,2,3)$:
\begin{equation}
\label{eqn:GH-forms}
\Omega_i=(d\theta +\eta )\wedge dx_i+V\, dx_j\wedge dx_k,
\end{equation}
and the following one for the hyperk\"ahler metric:
\begin{equation}
\label{eqn:GH-metric}
g= V^{-1}\cdot (d\theta +\eta )^2+V\cdot (dx_1^2+dx_2^2+dx_3^2).
\end{equation}
Here the forms $dx_1,dx_2,dx_3$ are a basis for the annihilator of
$\partial_\theta$, and so $(x_1,x_2,x_3)$ are (at least locally)
coordinates for the orbit space of $\partial_\theta$. The function
$V$ satisfies $\partial_\theta V\equiv 0$ and is thus locally a
function of only $(x_1,x_2,x_3)$. The $1$-form $\eta$ annihilates
$\partial_\theta$ and is invariant under its flow, so it is locally
pulled back from the orbit space of~$\partial_\theta$. All this
means that $\eta$ and $V$ are locally objects on
$(x_1,x_2,x_3)$-space, and we shall treat them as such for the purpose of
local calculations.

The projection onto the orbit space of $\partial_\theta$ is a
Riemannian submersion from the metric $(1/V)\,
g=g(\partial_\theta ,\partial_\theta )\, g$ to the Euclidean
metric $dx_1^2+dx_2^2+dx_3^2$. The condition for (\ref{eqn:GH-forms})
to define a triple of closed $2$-forms is then
\begin{equation}
\label{eqn:curl}
d\eta +*dV=0,
\end{equation}
where the Hodge star operator is in terms of that Euclidean metric and the
orientation defined by $(dx_1,dx_2,dx_3)$. For the (local) existence
of $\eta$ it is necessary and sufficient that $V$ be harmonic with
respect to the Euclidean metric. We call $V$ the
{\bf Gibbons--Hawking potential}.

The systematic construction of a hyperk\"ahler structure
on $M\times S^1_\theta$ proceeds as follows.
Start with a local diffeo\-mor\-phism $x=(x_1,x_2,x_3)\co M\to \R^3$,
and consider the metric $x^*g_{\R^3}$, where $g_{\R^3}$ is the
standard Euclidean metric on $\R^3$. Use $x$ also to pull the
standard orientation of $\R^3$ back to $M$. Let now $\eta$ and $V$
be a $1$-form and a function, respectively, defined on $M$ and satisfying
(\ref{eqn:curl}) with respect to the metric $x^*g_{\R^3}$ and the
pulled-back orientation. By lifting $x,\eta ,V$ to $M\times
S^1_\theta$ in the obvious way, and inserting them into
the defining equations (\ref{eqn:GH-forms})
and~(\ref{eqn:GH-metric}), we obtain a hyperk\"ahler structure on
$M\times S^1_\theta$ invariant under the flow of~$\partial_\theta$.

One may regard $(M,x^*g_{\R^3})$ as a {\em non-schlicht} domain in
$\R^3$, and $\eta$ and $V$ as multiple-valued objects in $\R^3$. On
$M$, however, they are perfectly well defined, and so
equation~(\ref{eqn:curl}) can be read as an identity on~$M$.

We now restrict our attention to $3$-manifolds $M$ of the form
$M=\Sigma \times\R_t$ with $\Sigma$ an open surface, and impose both
conditions (i) and~(ii). The following definition is useful for
describing the special features of this case.

\begin{defn}
A tensorial object {\bf o} (on $W=M\times S^1_{\theta}$ or on~$M$)
is called {\bf homogeneous of degree $k$}
if $L_{\partial_t}{\bf o}=k\cdot{\bf o}$.
\end{defn}

Condition (ii) requires that $g$ and the symplectic forms be
homogeneous of degree~$1$. Since $\partial_\theta$ is invariant
under the flow of $\partial_t$, the potential
$V=1/g(\partial_\theta ,\partial_\theta )$ must be homogeneous of
degree~$-1$. Condition (ii) provides a convenient choice for the
Hamiltonian functions~$x_i$, because one easily checks that
$x_i:=\Omega_i (\partial_\theta ,\partial_t)$ satisfies
$dx_i=\partial_\theta\ip\Omega_i$ in this case. This choice has
the virtue that the $x_i$ are homogeneous of degree~$1$, i.e.\
that $\partial_tx_i=x_i$. We call
$x:=(x_1,x_2,x_3)\co\Sigma \times\R_t\to\R^3\setminus\{ 0\}$ the
{\bf momentum map}. Then the equations $\partial_tx_i=x_i$ say
that $\partial_t$ is $x$-related to the position vector field on~$\R^3$.

The uniqueness of $\eta$ for given Hamiltonian functions $x_i$
implies that in the present situation both $\eta$ and the function $\eta
(\partial_t)$ are homogeneous of degree~$0$. Consider now the
function $\rho :=|x|\co\Sigma \times \R_t\to\R^+$. The $1$-form
$\frac{d\rho}{\rho}$ is homogeneous of degree~$0$ and satisfies
$\frac{d\rho}{\rho}\, (\partial_t)\equiv 1$. Hence
\[ \eta = \eta (\partial_t)\, \frac{d\rho}{\rho}+\xi ,\]
where the $1$-form $\xi$ is homogeneous of degree $0$ (i.e.\
$\partial_t$-invariant), $\partial_{\theta}$-invariant, and it
annihilates both $\partial_t$ and~$\partial_{\theta}$. So $\xi$ is the
pull-back of a $1$-form on~$\Sigma$, for which we continue to
write~$\xi$.

The map $x/|x|\co\Sigma\times\R_t\to S^2$ is
independent of $t$ and thus the pull-back of a unique map
$\Phi\co\Sigma\to S^2$. The latter is a local diffeomorphism. We endow
the unit sphere $S^2$ with the metric induced from the standard
metric on $\R^3$, and with the orientation as boundary of the
$3$-ball. This determines on $S^2$ a holomorphic structure and a
standard volume form~$\mbox{\rm Vol}_{S^2}$. We endow $\Sigma $ with
the holomorphic structure $J$ lifted from $S^2$ by $\Phi$, i.e.\
the structure that turns $\Phi\co\Sigma\to S^2$ into a local
biholomorphism.

\begin{thm}[The homogeneous Gibbons--Hawking ansatz]
\label{thm:GH}
The function $V$ and the $1$-form $\eta = \eta
(\partial_t)\, \frac{d\rho}{\rho}+\xi$ satisfy (\ref{eqn:curl}) on
$\Sigma\times\R_t$ if and only if the following two conditions are
satisfied:
\begin{itemize}
\item $d\xi = \rho V\,\Phi^*\mbox{\rm Vol}_{S^2}$,
\item $\varphi :=\eta (\partial_t)+\bi\rho V$ is the
pullback to $\Sigma\times\R_t$ of a holomorphic function
on~$(\Sigma ,J)$.
\end{itemize}
\end{thm}

\begin{proof}
Let $\tilde{u}+\bi\tilde{v}$ be a local holomorphic
coordinate on $S^2$, and let $u+\bi v$ be its pullback under
$\Phi$. Then $u,v$ are homogeneous of degree $0$ and $(\rho
,u,v)$ are local coordinates on $\Sigma\times\R_t$ giving the same
orientation as $(x_1,x_2,x_3)$.

Since $V$ is homogeneous of degree~$-1$, we have $V_\rho
=-\rho^{-1}V$, therefore
\[ dV=-\rho^{-1}V\, d\rho +V_u\, du +V_v\, dv.\]
On the other hand,
\[ *d\rho =\rho^2\,\Phi^*\mbox{\rm Vol}_{S^2},\quad
*du=-d\rho\wedge dv,\quad *dv=d\rho\wedge du,\]
and so
\[ *dV= -\rho V\,\Phi^*\mbox{\rm Vol}_{S^2} -
V_u\, d\rho\wedge dv +V_v\, d\rho\wedge du.\]
Since $\xi$ is homogeneous of degree~$0$ and annihilates $\partial_t$,
we have $\xi =\xi_1(u,v)\, du+\xi_2(u,v)\, dv$ and
\[ d\eta = (\xi_{2u}-\xi_{1v})\, du\wedge dv-
\eta (\partial_t)_v\,\frac{d\rho}{\rho}\wedge dv -\eta
(\partial_t)_u\,\frac{d\rho}{\rho}\wedge du.\]
Then (\ref{eqn:curl})
is seen to be equivalent to the system
\begin{eqnarray*}
d\xi                & = & \rho V\,\Phi^*\mbox{\rm Vol}_{S^2}, \\
\eta (\partial_t)_u & = & \;\;\;\rho V_v  \; = \;\;\;\: (\rho V)_v, \\
\eta (\partial_t)_v & = & -\rho V_u       \; = \; -(\rho V)_u.
\end{eqnarray*}
The last two equations are the Cauchy--Riemann equations
for~$\varphi$.
\end{proof}

In order to describe the systematic construction of hyperk\"ahler
structures satisfying (i) and (ii), it is convenient to fix the
complex structure $J$ on $\Sigma$ in advance. We then use {\em
holomorphic data\/} of the following kind
to construct the desired structures:
\begin{itemize}
\item a local biholomorphism $\Phi\co (\Sigma ,J)\to
S^2$,
\item a holomorphic function $\varphi\co (\Sigma ,J)\to\Hy$
with values in the upper half-plane~$\Hy$.
\end{itemize}

Being homogeneous of degree~$1$, the function $\rho$ must be of the
form $\rho=e^t\rho_0$ with $\rho_0$ the pullback of a positive
function on~$\Sigma$. Given a choice of $\Phi , \varphi ,\rho_0$, the
construction is as follows. The momentum map is given by
$x=\rho\,\Phi$. We take an antiderivative $\xi$ for $(\mbox{\rm
Im}\,\varphi )\,\Phi^*\mbox{\rm Vol}_{S^2}$ on $\Sigma$. Set $\eta
=(\mbox{\rm Re}\,\varphi )\, \frac{d\rho}{\rho} +\xi\;$ and $\;
V=(\mbox{\rm Im}\,\varphi)\,\rho^{-1}$, both pulled back to~$W$. The
hyperk\"ahler structure is given by (\ref{eqn:GH-forms})
and~(\ref{eqn:GH-metric}) with these values for~$x,\eta ,V$.

Given any pair of functions $h_1,h_2\co\Sigma\to\R$, the
diffeomorphism of $W$ given by
\[ \bigl( p,t,\theta \bigr)\longmapsto \bigl( p,t+h_1(p),\theta +h_2(p)\bigr)\]
preserves $\partial_t$ and $\partial_{\theta}$ and pulls the
hyperk\"ahler structure with data $(\Phi ,\varphi ,\rho_0 ,\xi)$
back to the one corresponding to
$(\Phi ,\varphi ,e^{h_1}\rho_0 ,\xi +dh_2)$.
This implies that if $\Sigma$ is simply-connected, then the triple
\[ \bigl(\mbox{\rm hyperk\"ahler structure},\partial_\theta ,
\partial_t\bigr) \]
is determined up to isomorphism by the holomorphic data $(\Phi ,\varphi )$
alone, thus allowing us to make any choice for $\rho_0$
and~$\xi$. For general $\Sigma$, the choice of $\xi$ matters, but
the function $\rho_0$ can always be chosen freely. We shall presently
establish a convenient choice for~$\rho_0$.
\section{Canonical metrics}
\label{section:canonical}
We continue to consider structures satisfying (i) and (ii). Recall
that the canonical slice is the hypersurface
\[ S=\{p\in W\co g(\partial_t,\partial_t )_p=1\}.\]
The vector field $\partial_{\theta}$ is a Killing field for $g$
and commutes with~$\partial_t$. So $\partial_{\theta}$ is tangent
to~$S$, and the restriction of $\partial_{\theta}$ to $S$ is a Killing
field for the $3$-dimensional metric $g_3$ induced on $S$ by
the hyperk\"ahler metric. That metric~$g_3$, in turn, induces
the quotient metric on
$S/(\mbox{\rm flow of}\;\partial_\theta )$ which makes the
projection a Riemannian submersion.

We now introduce the holomorphic function $\psi =-1/\varphi$, which still
takes values in the upper half-plane~$\Hy$.
The next lemma shows that the choice $\rho_0=
\mbox{\rm Im}\,\psi$ is especially convenient, because then $\{
t=0\}$ is the canonical slice.

\begin{lem}
\label{lem:slice}
The canonical slice is the product $G\times S^1_\theta$, where
$G$ is the surface in $\Sigma\times\R_t$ described
equivalently by any of the following equations:
\[
\begin{array}{ccl}
V    & = & |\varphi |^2 ,\\
\rho & = & \mbox{\rm Im}\,\psi ,\\
t    & = & \log \mbox{\rm Im}\,\psi -\log\rho_0.
\end{array} \]
\end{lem}

\begin{proof}
Formula (\ref{eqn:GH-metric}) gives
\[ g(\partial_t,\partial_t)=V^{-1}\cdot(\eta (\partial_t))^2
+V\cdot (x_1^2+x_2^2+x_3^2).\]
By Theorem~\ref{thm:GH}, this can be written as
\[ g(\partial_t,\partial_t)=V^{-1}\cdot
(\hbox{\rm Re}\,\varphi )^2+V\cdot\rho^2 .\]
So the canonical slice
is given by the equation $V^{-1}\cdot (\hbox{\rm Re}\,\varphi)^2
+V\cdot\rho^2=1$, which again by Theorem~\ref{thm:GH} transforms to
\[ V=(\hbox{\rm Re}\,\varphi )^2+(\hbox{\rm Im}\,\varphi )^2
=|\varphi |^2 .\]
Using $\rho =(\rho V)/V$ and Theorem~\ref{thm:GH}, this is
seen to be equivalent to
\[ \rho =\frac{\rho V} {|\varphi |^2}
=\frac{\hbox{\rm Im}\,\varphi}{|\varphi |^2}
= \hbox{\rm Im}\,\psi .\]
The third description of the canonical slice then follows from
$\rho =e^t\rho_0$.
\end{proof}

The surface $G$ is the graph of a function $\Sigma\to\R_t$. Hence
the map $(p,t,\theta )\mapsto p$ induces a diffeomorphism $\sigma\co
S/(\mbox{\rm flow of}\;\,\partial_\theta )\to\Sigma$.

\begin{lem}
The diffeomorphism $\sigma$ sends the quotient metric to the
metric
\begin{equation}
\label{eqn:g-sigma}
g_\Sigma :=\frac{1}{|\psi |^2}\,\left[ (d\,\hbox{\rm Im}\,\psi)^2
+(\hbox{\rm Im}\,\psi )^2\,\Phi^*g_{S^2} \right] 
\end{equation}
on~$\Sigma$,
with $g_{S^2}$ denoting the standard metric on $S^2$.
\end{lem}

\begin{proof}
It follows from (\ref{eqn:GH-metric}) that the orthogonal
complement to $\partial_\theta$ is described by the equation
$d\theta +\eta =0$, both
on $W$ and on $S$. The restriction of the hyperk\"ahler metric to
this complement is then the same as the restriction of the quadratic
form $q:=V\cdot (dx_1^2+dx_2^2+dx_3^2)$. This $q$ is
$\partial_{\theta}$-invariant and has $\partial_\theta$ as an
isotropic direction, therefore its restriction to the canonical
slice $S$ is the pullback of the quotient metric under the quotient
projection.

The standard metric of $\R^3$ equals $dr^2 +r^2\, g_{S^2}$ in
spherical coordinates. This and the equations $\rho =\mbox{\rm
Im}\,\psi$ and $V=|\varphi |^2$ from Lemma~\ref{lem:slice} imply that
the restriction of $q$ to $S$ is given by the right-hand side of
(\ref{eqn:g-sigma}), with $\Phi\co\Sigma\to S^2$ replaced
by~$x/\rho\co S\to S^2$. It is then clear that $\sigma$ pulls
$g_\Sigma$ given by (\ref{eqn:g-sigma}) back to the quotient metric.
\end{proof}
\section{A slice-complete example}
\label{section:example}
The orbits of the Killing field $\partial_\theta$ in the canonical
slice are circles (of variable length $2\pi V^{-1/2}=2\pi\, |\psi
|$). So it is clear that the induced metric $g_3$ is complete if
and only if the quotient metric is complete; the latter in turn is
isometric to~$g_\Sigma$. In the sequel the Riemann surface
$(\Sigma ,J)$ will be the unit disc $\D\subset\C$.

\begin{thm}
\label{thm:example}
There are holomorphic data~$(\Phi ,\psi )$ on $\D$
such that formula~(\ref{eqn:g-sigma}) defines a complete metric
$g_{\D}$ on~$\D$. Thus, the pair $(\Phi ,\psi )$ gives rise to a
slice-complete hyperk\"ahler metric on $\D\times\R_t\times S^1_{\theta}$.
\end{thm}

The construction of the pair $(\Phi ,\psi )$ will take up the rest
of this section.

Saying that $g_{\D}$ is complete means that any proper path
$\gamma\co [0, +\infty )\rightarrow{\D}$ has infinite length in this
metric. Intuitively, for the metric $\Phi^* g_{S^2}$ to be
complete the map $\Phi\co \D\to S^2$ would have to wrap $\D$ around $S^2$
so as to push the boundary of $\D$ infinitely far away. Such a map,
however, would have to be a covering. The $2$-sphere being simply connected,
this is impossible. Still, {\em most} of that boundary can be pushed
infinitely far away. This is achieved as follows.

Equip $\D$ with the Poincar\'e metric and consider the conformal
covering projection
\[ \Phi \co {\D}\longrightarrow  S^2\setminus\{ p_1,p_2,p_3\} \]
of the sphere with three punctures by the unit disc. Let us recall
the construction of such a covering map. Take a hyperbolic
triangle with its three vertices on $\partial{\D}$ (a so-called
{\em ideal} hyperbolic triangle),
tessellate $\D$ by this triangle and the infinitely many
images under successive reflection on the sides (Figure~\ref{figure:tess}),
then consider
the quotient ${\D}/\Gamma$ where $\Gamma$ is the group consisting of
the hyperbolic translations within the group generated by those
reflections. This quotient space ${\D}/\Gamma$ is the result of gluing
corresponding sides of
two copies of the original triangle, hence
a sphere with three punctures. As a consequence of the uniformisation
theorem, there is a biholomorphism
\[ \D /\Gamma\longrightarrow
S^2\setminus\{ p_1,p_2,p_3\} ; \]
see \cite{thur97}, in particular pp.~117 and 147--149.

\begin{figure}[h]
\centering
\includegraphics[scale=0.3]{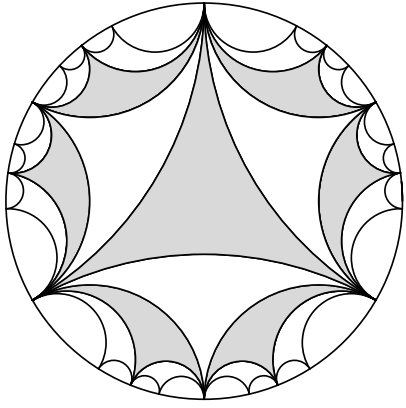}
  \caption{The hyperbolic tessellation of $\D$ corresponding to~$\Phi$.}
  \label{figure:tess}
\end{figure}

The map $\Phi$ is the composition
\[ \D\longrightarrow\D /\Gamma\longrightarrow
S^2\setminus\{ p_1,p_2,p_3\} , \]
and it has rank $2$ everywhere.

There is a distinguished sequence $(z_n)$ of points on the unit
circle $\partial{\D}$, namely the vertices of all triangles in the
tessellation. For $j=1,2,3$ choose a small closed metric ball $B_j$,
centred at the puncture $p_j$ in $S^2$, such that the balls
$2B_1,2B_2,2B_3$ of double radius have disjoint closures.
Their preimages under $\Phi$ are shown in Figure~\ref{figure:horo}.

For each $n$ let $D_n$ be the connected component of
$\Phi^{-1}(B_1\cup B_2\cup B_3)$ having $z_n$ as limit point, and
let $2D_n$ be the same for $\Phi^{-1}(2B_1\cup 2B_2\cup 2B_3)$.
Then  both $(D_n)$ and $(2D_n)$ form a sequence of pairwise
disjoint horodisc-like regions.

\begin{figure}[h]
\centering
\includegraphics[scale=0.3]{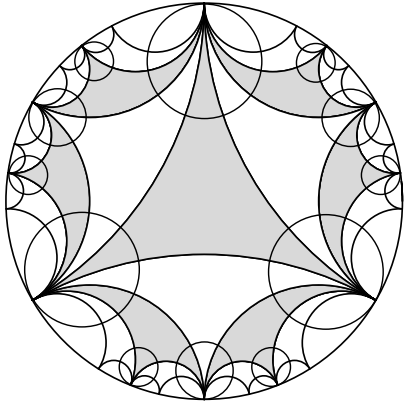}
  \caption{The hyperbolic tessellation and the horodisc-like regions.}
  \label{figure:horo}
\end{figure}

\begin{lem}
\label{lem:longpath}
Let $\gamma \co[0,+\infty )\rightarrow\D$ be any
proper path. Either $\gamma$ has infinite length in the metric
$\Phi^* g_{S^2}$, or it has an end $\gamma ([t_0,+\infty ))$
contained in some region $2D_{n_0}$.
\end{lem}

\begin{rem}
If $\gamma$ is proper in $\D$ and contained in
$2D_{n_0}$, we must of course have $\gamma (t)\to z_{n_0}$ as $t\to
+\infty$. So the lemma says that a path of finite length in
the metric $\Phi^*g_{S^2}$ can only escape the
disc through one of the points $z_n$. Intuitively, the covering map
$\Phi$ wraps around the punctured sphere so as to push the boundary
of the disc infinitely far away, with the exception of
the countable set $\{ z_n\co n\in\N\}$.
\end{rem}

\begin{proof}
Suppose first that the path visits
$\cup_{n=1}^{\infty} D_n$ only finitely often. By deleting an
initial segment, we may then assume that the trace of $\gamma$ is disjoint
from all the~$D_n$.

Each triangle $T_k$ of the tessellation intersects exactly three of
the $D_n$. Let $T'_k$ be the compact region obtained by
removing from $T_k$ the interior of those three intersections.
With the metric $\Phi^*g_{S^2}$ the $T'_k$ are pairwise congruent
spherical regions in the shape of a hexagon, with three sides
coming from the boundaries $\Phi^{-1}(\partial B_j)$, $j=1,2,3$
(call these the odd sides), and three sides coming from the sides
of $T_k$ (call these the even sides). The path $\gamma$ is proper
and contained in the union $\cup_{k=1}^{\infty}T'_k$, therefore it
must visit infinitely many different regions $T'_k$, and among
those there must be infinitely many $T'_k$ where $\gamma$ enters
on one (necessarily even) side and exists on another (also even)
side. If $c_1>0$ is the minimum spherical distance between even
sides, which is the same for all $T'_k$, then the length of
$\gamma$ in the metric $\Phi^* g_{S^2}$ is bounded from below by
the series $c_1+c_1+\cdots$ and therefore infinite.

Suppose now that $\gamma$ visits the disjoint union
$\cup_{n=1}^{\infty} D_n$ infinitely often and there are at least
two regions $D_{n_1}$ and $D_{n_2}$ which the path visits
infinitely many times. Then the length of $\gamma$ in the metric
$\Phi^* g_{S^2}$ is bounded below by the series $c_2+c_2+\cdots$,
where $c_2$ is the minimum spherical distance between the balls
$B_1,B_2,B_3$, and hence infinite.

There remains the case when $\gamma$ visits but a single region
$D_{n_0}$ infinitely often. If it has an end $\gamma ([t_0,+\infty
))$ contained in $2D_{n_0}$, we are done. Otherwise $\gamma$ takes
infinitely many journeys from $\partial D_{n_0}$ to $\partial
(2D_{n_0})$ and the length of $\gamma$ in the metric $\Phi^*
g_{S^2}$ is bounded below by the series $c_3+c_3+\cdots$, where
$c_3$ is the radius of the metric ball $B_j=\Phi (D_{n_0})$, and
so again this length is infinite.
\end{proof}

A path in $2B_j\setminus\{ p_j\}$ may well have
finite spherical length and limit $p_j$. Thus $2D_{n_0}$ contains
paths with $z_{n_0}$ as limit and finite length in the metric
$\Phi^*g_{S^2}$. To ensure that $g_\D$ be a complete metric, we need
to choose the function $\psi$ so that the integral of $|d\,\hbox{\rm
Im}\,\psi |/|\psi |$ along such paths is infinite. Such a $\psi$ can
be found with the help of so-called Blaschke products.
For the reader's convenience we first review this notion.

\vspace{2mm}

Given $a\in{\D}$ with $a\neq 0$, let $F_a(z)$ be the
orientation-preserving isometry for the Poincar\'e metric
that exchanges $0$ with $a$. The
line segment joining 0 to $a$ is a Poincar\'e geodesic, and $F_a$ is the
$180^{\hbox{\scriptsize o}}$ Poincar\'e rotation about the
Poincar\'e midpoint of that segment (Figure~\ref{figure:rotation}).
It is given by
$\displaystyle{F_a(z)=\frac{a-z}{1-\overline{a}\, z}}$.

\begin{figure}[h]
\labellist
\small\hair 2pt
\pinlabel $0$ [br] at 207 222
\pinlabel $a$ [br] at 352 294
\endlabellist
\centering
\includegraphics[scale=0.3]{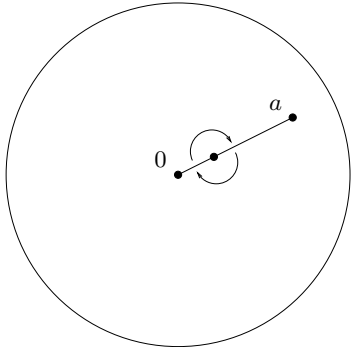}
  \caption{Poincar\'e rotation.}
  \label{figure:rotation}
\end{figure}

The {\bf Blaschke factor} $B_a(z)$ is the function
\[ B_a(z)=\frac{\overline{a}}{|a|}\,\frac{a-z}{1-\overline{a}\,
z}\, ,\]
i.e.\ the result of multiplying the isometry $F_a$ by a unitary
constant so that $B_a(0)$ is a positive real number. The purpose
of this normalisation is to get a simple convergence condition
for {\bf Blaschke products}, which are finite or infinite
products of the form
\[ B(z)=z^m\cdot\prod_k B_{a_k}(z).\]
It is well known \cite[Section 6.7]{haym64}
that under the condition $\sum_k
(1-|a_k|)<\infty$ this product converges to a holomorphic function
$B\co\D\rightarrow\D$.

\begin{thm}[{\cite[Theorem 1]{bcp83}}]
Let $(z_n)$ be a
sequence of pairwise distinct points on $\partial{\D}$. Let
$D\subset {\D}$ be a convex open domain with its boundary
interior to $\D$ except for the point $1\in\partial{\D}$.
Then there exists a Blaschke product $B\co\D\rightarrow\D$
such that for each $n$
we have a finite number $\lambda_n$ with
\[ \frac{B(z)-1}{z-z_n}\longrightarrow \lambda_n \;\;\hbox{as}\;\; z
\longrightarrow z_n,\;\; z\in z_n D:= \{\, z_n\zeta\co\zeta\in D\,\}.\qed\]
\end{thm}

We are now going to apply this theorem to the sequence $(z_n)$ of vertices
in our hyperbolic tessellation.

In the Euclidean metric on~$\D$, the size of the regions $2D_n$
is bounded above by the size of the three of those horodisc-like
regions at the vertices of the hyperbolic triangle containing
the origin $0\in\D$. Thus we
can choose a circular disc $D\subset{\D}$ whose boundary is
interior to $\D$ except for 1, and such that for each $n$ the
rotated image $z_n D$ contains~$2D_n$. By the theorem above
we have a Blaschke product $B\co\D\to\D$ such that for all $n$
\[ B(z)\longrightarrow 1 \;\;\hbox{as}\;\; z\longrightarrow
z_n,\;\; z\in 2D_n.\]

The holomorphic function $1-B$ maps $\D$ into
the open disc of radius $1$ and
centre~$1$. Moreover, for each $n$ this function has limit $0$ as
$z$ approaches $z_n$ inside $2D_n$.

The open disc of radius $1$ and centre $1$ is contained in the
right half-plane $\{ z\co\mbox{\rm Re}\, z>0\}$, on which there
is a well-defined holomorphic square root with values in the
quadrant $Q_1:=\{ z\co\hbox{arg}\, z\in (-\pi/4,\pi /4)\}$.
So there is a holomorphic function $\sqrt{1-B}\co\D\rightarrow Q_1$
mapping $\D$ into the region bounded by a half lemniscate
symmetric about the positive real axis. This function likewise has
limit $0$ as $z$ approaches $z_n$ inside $2D_n$.

Finally, we set $\psi=\bi\cdot\sqrt{1-B}$. This holomorphic function
maps $\D$ into the quadrant
\[ Q_2=\{ z\co \hbox{arg}\, z\in
(\pi/4,3\pi /4)\} = \{ u+\bi\, v\co |u|<|v|\}\subset\Hy .\]
In fact, $\psi ({\D})$ is
contained in the region bounded by a half lemniscate symmetric
about the positive imaginary axis (Figure~\ref{figure:map0}).
Moreover, $\psi (z)$ converges to $0$ as $z$ approaches $z_n$ inside $2D_n$.

\begin{figure}[h]
\labellist
\small\hair 2pt
\pinlabel $\psi$ [b] at 270 121
\pinlabel $z_n$ [tl] at 161 15
\pinlabel $2D_n$ [b] at 148 67
\pinlabel $Q_2$ at 563 164
\pinlabel {$\psi(\D )$} [b] at 468 153
\endlabellist
\centering
\includegraphics[scale=0.5]{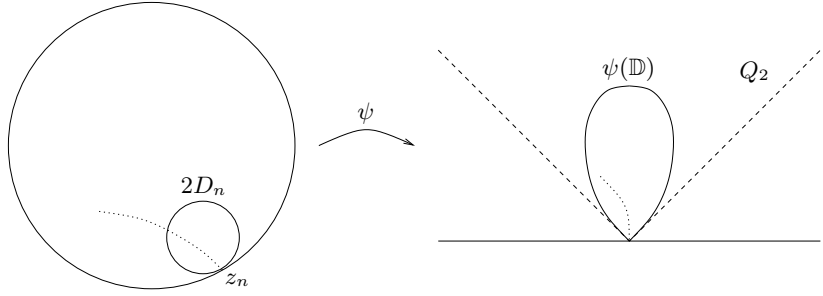}
  \caption{The map $\psi$.}
  \label{figure:map0}
\end{figure}

This finishes the construction of the pair $(\Phi ,\psi )$. We are now going
to verify that the metric $g_{\D}$ defined by this pair as in
equation~(\ref{eqn:g-sigma}) is indeed complete. This will
conclude the proof of Theorem~\ref{thm:example}.

Since $\psi$ satisfies $|\hbox{Re}\,\psi |<|\hbox{Im}\,\psi |$, we
have $|\hbox{Im}\,\psi |/|\psi |> 1/\sqrt{2}$. By
Lemma~\ref{lem:longpath}, a proper path in $\D$ is
infinitely long in the metric
$(\hbox{Im}\,\psi /|\psi |)^2\cdot\Phi^* g_{S^2}$,
hence also in the metric $g_{\D}$, unless it has an end contained in
some region $2D_{n_0}$.

But if a proper path $\gamma\co [t_0,+\infty )\rightarrow\D$
is contained in $2D_{n_0}$, then we must have
\[ \gamma (t)\longrightarrow z_{n_0}\;\;\hbox{and}\;\;\psi(\gamma (t))
\longrightarrow 0 \;\;\hbox{as}\;\; t\longrightarrow +\infty .\]
Writing $\psi(\gamma (t))=u(t)+\bi\, v(t)$, we have
\[ \int_{\gamma}\frac{|d\,\hbox{Im}\,\psi |}{|\psi |}>
\int_{t_0}^{+\infty}\frac{|v'(t)|}{\sqrt{2}\, v(t)}\, dt =
\frac{1}{\sqrt{2}}\,\bigl(\hbox{total variation of}\;\log v(t)\bigr) . \]
Since $\log v(t)\to\log 0= -\infty$ as $t\to +\infty$, we
deduce that the integral of $|d\,\hbox{Im}\,\psi |/|\psi |$ along
$\gamma$ is infinite, and so is the length of $\gamma$ in the
metric~$g_{\D}$.
\section{An infinite-dimensional family of examples}
\label{section:many}
The preceding construction goes through if we replace $\psi$ by
$\psi_\mu =\mu\circ\psi$, where $\mu$ is any holomorphic function
whose domain contains $\{ 0\}\cup\psi (\D )$ and such that $\mu
(0)=0$ and $\mu (\psi (\D ))\subset Q_2$, see Figure~\ref{figure:map1}.
With $\Phi$ the covering map described in Section~\ref{section:example},
we thus obtain an infinite-dimensional family of holomorphic data
$(\Phi ,\psi_{\mu})$ that yield slice-complete hyperk\"ahler
structures in $\D\times\R_t\times S^1_\theta$.

\begin{figure}[h]
\labellist
\small\hair 2pt
\pinlabel $\psi_{\mu}$ [b] at 270 121
\pinlabel {$\psi_{\mu}(\D )$} [b] at 487 166
\endlabellist
\centering
\includegraphics[scale=0.5]{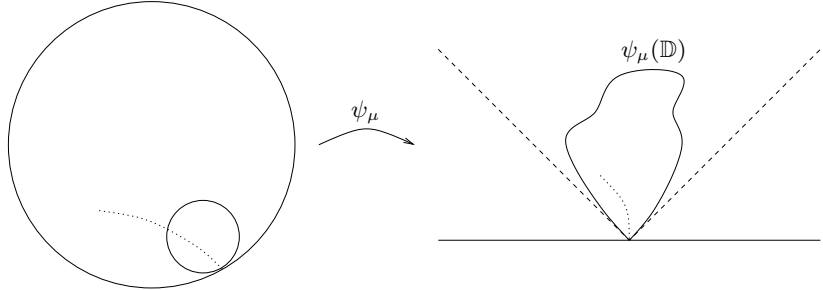}
  \caption{The map $\psi_{\mu}$.}
  \label{figure:map1}
\end{figure}

We now want to show that the family of pairs $(\Phi ,\psi_{\mu})$
gives rise to an infinite-dimensional family of slice-complete
hyperk\"ahler metrics.
Consider triples $({\mathcal H},X,Y)$ made of a hyperk\"ahler
structure ${\mathcal H}=(g,J_1,J_2,J_3)$, a vector field $X$
preserving the entire structure, and a vector field $Y$ such that
$L_Y\Omega_i=\Omega_i$ and $[X,Y]\equiv 0$. If two such triples are
isomorphic, then the image $x(S)$ of the canonical slice of $Y$ under the
momentum map is the same for both. For a triple in our family, the
image $x(S)$ is the radial graph in $\R^3$ of the multi-valued
function $(\mbox{\rm Im}\, \psi_\mu )\circ\Phi^{-1}\co S^2\rightarrow
\R^+$. As $\mu$ varies, these graphs form an infinite-dimensional family of
immersed surfaces in $\R^3$, so we have an infinite-dimensional
family of isomorphism classes of triples. Then the metrics also constitute
an infinite-dimensional family of isometry classes, because for a given
hyperk\"ahler metric $g$ the number of degrees of freedom for
$J_1,J_2,J_3,X,Y$ is bounded {\em a priori} by a finite constant: the
complex structures $J_1,J_2,J_3$ are parallel with respect to the
Riemannian connection; the Lie algebra of Killing vector fields
is finite-dimensional~\cite[Thm.~VI.3.3]{kono63}; any two
homothetic vector fields differ by a Killing field.
\section{Non-flatness}
\label{section:non-flat}
The following proposition implies, as we shall see below,
that the hyperk\"ahler metrics in our examples are non-flat.
This proposition may be regarded as a converse of~\cite[Thm.~10]{gego08}.

\begin{prop}
\label{prop:non-flat}
If a metric with a homothetic vector field is slice-complete and
flat, then the canonical slice is compact.
\end{prop}

\begin{proof}
Consider a tubular neighbourhood $U$ of the canonical slice $S$. Endow the
universal cover $\widetilde{U}$ with the lifted flat metric.
The lift of the homothetic vector field to $\widetilde{U}$ is a vector field
$\widetilde{Y}$ homothetic for the lifted metric, and its canonical
slice is the inverse image $\widetilde{S}$ of $S\subset U$ under the
covering map $\widetilde{U}\to U$. This $\widetilde{S}$ is connected,
because its tubular neighbourhood $\widetilde{U}$ is connected.
We conclude that $\widetilde{S}$ is
a universal cover of $S$ under the restricted projection
$\widetilde{S}\to S$, which is a local isometry for the induced
metrics. Since $S$ is complete with the induced metric, so
is~$\widetilde{S}$.

There is a local isometry $F\co \widetilde{U}\to\E^n$ into Euclidean
space. If $U'\subset\widetilde{U}$ is an
open domain on which $F$ is injective, then on the image $F(U')$ we
have the vector field $F_*\widetilde{Y}$ that is homothetic for the
Euclidean metric of $\E^n$ restricted to $F(U')$. Then
there is a vector field $Y_0$, defined on all of $\E^n$ and
homothetic for the Euclidean metric, that coincides with
$F_*\widetilde{Y}$ on $F(U')$. On $\widetilde{U}$ we have the
homothetic vector fields $\widetilde{Y}$ and $F^*Y_0$, and they have
to be identical because they coincide on~$U'$. This means that the
local isometry $F$ sends $\widetilde{Y}$ to $Y_0$, hence it maps the
canonical slice $\widetilde{S}$ of $\widetilde{Y}$ to the canonical
slice $S_0$ of $Y_0$. Since $\widetilde{S}$ is complete, the
restriction $F\co\widetilde{S}\to S_0$ must be a covering. But
homothetic vector fields in $\E^n$ are linear, and $S_0$ is
thus an ellipsoid. This forces $\widetilde{S}$ to be diffeomorphic
with $S^{n-1}$ and in particular compact. {\em A fortiori}, the slice $S$
must be compact.
\end{proof}

Hyperk\"ahler metrics are Ricci flat~\cite[p.~284]{bess87}, and thus
in particular Einstein metrics.
Regularity results of DeTurck--Kazdan~\cite{deka81},
cf.~\cite[Sections~5.E--F]{bess87}, say that Einstein metrics
(in dimension at least~$3$) are real analytic in harmonic
coordinates (the reason being that the equations for an Einstein metric
form a quasi-linear elliptic system in such coordinates).
Thus, if such a metric is flat in some domain, then it
must be flat everywhere. With Proposition~\ref{prop:non-flat}
we conclude that the hyperk\"ahler
metrics described in Sections \ref{section:example} and~\ref{section:many}
are non-flat on every open set, because the canonical slice of
each of these metrics is complete but non-compact.

According to \cite[Prop.~34]{gego08}, $4$-dimensional
hyperk\"ahler metrics that are cone metrics are necessarily flat,
so the metrics constructed above are not
cone metrics in any open set, i.e.\ none of their homothetic vector
fields are hypersurface orthogonal, not even locally.
\section{A Carnot--Carath\'eodory phenomenon}
\label{section:CC}
In this section we show how the examples constructed in Sections
\ref{section:example} and~\ref{section:many} lead in a natural way to
incomplete Riemannian metrics that nonetheless induce complete
Carnot--Carath\'eodory distances. We begin by explaining these
concepts.

A $1$-form $\alpha$ on a $3$-manifold $M$ is a {\bf contact form} if
$d(e^t\alpha )$ is a symplectic form on $M\times\R_t$. This is
equivalent to $\alpha\wedge d\alpha$ being a volume form on~$M$. The pair
$(M,\alpha )$ is then called a {\bf contact manifold}.

It is well known~\cite[Section~3.3]{geig08}
that any two points on a contact manifold
$(M,\alpha )$ can be joined by a {\bf Legendre path}, i.e.\ a path
$\gamma (s)$ such that $\alpha (\gamma'(s))\equiv 0$. Given a Riemannian
metric $g$ on~$M$, the induced {\bf Carnot--Carath\'eodory distance} is, for
each pair of points $p_1,p_2\in M$, the infimum of the lengths of
Legendre paths joining $p_1$ to~$p_2$, see \cite{mont02}
or~\cite{grom96}. Notice that this is bounded below by the
Riemannian distance, hence it is certainly a complete distance if
$g$ happens to be complete.

A triple of contact forms $(\alpha_1,\alpha_2,\alpha_3)$ on $M$ is
called a {\bf contact sphere} if for each ${\bf
c}=(c_1,c_2,c_3)\in\R^3\setminus \{ 0\}$ the linear combination
$\alpha_{\bf c}=c_1\alpha_1+c_2\alpha_2+c_3\alpha_3$ is a contact
form. The contact sphere is called {\bf taut} if all volume forms
$\alpha_{\bf c}\wedge d\alpha_{\bf c}$ with $|{\bf c}|=1$ are
equal, see~\cite{gego95}.
Multiplying the three $1$-forms of a taut contact sphere by the same
non-vanishing function $w$ gives a new  taut contact sphere, thanks to
the identity $w\alpha\wedge d(w\alpha )=w^2\,\alpha\wedge d\alpha$.

We now recall from \cite{gego08}
a correspondence between hyperk\"ahler structures
on $M\times\R_t$, satisfying condition (ii) of
Section~\ref{section:GH}, and taut contact spheres on~$M$. Given the
hyperk\"ahler structure, consider the symplectic forms
$\Omega_{\bf c}=c_1\,\Omega_1+c_2\,\Omega_2+c_3\,\Omega_3$ for
${\bf c}\in S^2$. We know that the volume form $\Omega_{\bf c}^2$
is the same for all unitary~$\bf c$. On the other hand, the equation
$L_{\partial_t}\Omega_{\bf c}=\Omega_{\bf c}$ is equivalent to
$\Omega_{\bf c}=d(\partial_t\ip\Omega_{\bf c})$, and this implies
\[ \partial_t\ip (\Omega_{\bf c}^2)=2\, (\partial_t\ip\Omega_{\bf c})\wedge
d(\partial_t\ip\Omega_{\bf c}) .\] Since the $3$-form
$\partial_t\ip (\Omega_{\bf c}^2)$ does not depend on $\bf c$, the
family $\bigl(
\partial_t\ip\Omega_{\bf c}\bigr)_{{\bf c}\in S^2}$ induces
a taut contact sphere on any transversal for~$\partial_t$. If
$(\alpha_1,\alpha_2,\alpha_3)$ is the contact sphere induced on the
transversal $\{ t=0\}$, the conditions
$L_{\partial_t}(\partial_t\ip\Omega_i)=\partial_t\ip\Omega_i$ and
$(\partial_t\ip\Omega_i)(\partial_t)=0$ lead to the expressions
\[ \partial_t\ip\Omega_i =e^t\alpha_i,\quad i=1,2,3,\]
where the $\alpha_i$ have been pulled back to $M\times\R_t$. Then
also $\Omega_i=d(e^t\alpha_i)$.

Conversely,
if we are given a taut contact sphere $(\alpha_1,\alpha_2,\alpha_3)$
on $M$, then --- subject only to the sign condition
$(\alpha_1\wedge d\alpha_1)/(\alpha_1\wedge\alpha_2\wedge\alpha_3)>0$,
which will be satisfied after a suitable permutation of the~$\alpha_i$ ---
the symplectic forms $\Omega_i=d(e^t\alpha_i)$ determine a
hyperk\"ahler structure on $M\times\R_t$ that obviously satisfies
condition (ii) of Section~\ref{section:GH}. These two
processes, passing from a hyperk\"ahler structure to a taut contact sphere
and vice versa, are inverses of each other.

Now we can define the natural metric mentioned at the beginning
of this section.

\begin{defn}
For a hyperk\"ahler structure satisfying condition (ii) of
Section~\ref{section:GH}, we write $(\omega_1,\omega_2,\omega_3)$
for the taut contact sphere induced by the $1$-forms $\partial_t\ip\Omega_i$
on the canonical slice and call $g_s:=\omega_1^2+\omega_2^2+\omega_3^2$
the {\bf short metric} on the canonical slice.
\end{defn}

Here is how the short metric $g_s$ and the canonical metric $g_3$ defined
in Section~\ref{section:canonical} are related. Any
taut contact sphere $(\alpha_1,\alpha_2,\alpha_3)$ on a
$3$-manifold satisfies the following structure equations, where
$(i,j,k)$ ranges over the cyclic permutations of $(1,2,3)$,
with $\beta_0$ a unique $1$-form and $\Lambda_0$ a unique function:
\begin{equation}
\label{eqn:structure}
d\alpha_i=\beta_0\wedge\alpha_i+\Lambda_0\,\alpha_j\wedge\alpha_k.
\end{equation}
If we have a hyperk\"ahler structure satisfying (ii) on $M\times\R_t$, and
$(\alpha_1,\alpha_2,\alpha_3)$ is the contact sphere induced by the
$1$-forms $\partial_t\ip\Omega_i$ on the transversal $\{ t=0\}$,
then the function $\Lambda_0$ from the structure equations is positive
and we have the formul\ae\
\[ \Omega_i=d(e^t\alpha_i)=e^t\bigl( \Lambda_0^{-1/2}(dt+\beta_0)\wedge
\Lambda_0^{1/2}\alpha_i+ \Lambda_0^{1/2}\alpha_j\wedge \Lambda_0^{1/2}
\alpha_k\bigl) .\]
It follows that the hyperk\"ahler metric is given by
\begin{equation}
\label{eqn:metric}
g=e^t\bigl( \Lambda_0^{-1}\cdot (dt+\beta_0 )^2+\Lambda_0\cdot
(\alpha_1^2+\alpha_2^2+\alpha_3^2)\bigr) .
\end{equation}
The canonical slice is thus given by the equation $t=\log\Lambda_0$.

The taut contact sphere $(\omega_1,\omega_2,\omega_3)$ induced on the
canonical slice $S$ is given by $\omega_i=e^t\alpha_i|_{TS}$.
With $\beta:=(dt+\beta_0)|_{TS}$ we have
\[ g_3:=g|_{TS}=\beta^2+\omega_1^2+\omega_2^2+\omega_3^2.\]
Clearly, $g_s$ is shorter than~$g_3$. The short metric
is incomplete in all the examples described in Sections
\ref{section:example} and~\ref{section:many}, because of the following result.

\begin{thm}[{\cite[Theorem 24]{gego08}}]
\label{thm:incomplete}
If a hyperk\"ahler structure on $\Sigma\times\R_t\times S^1_\theta$
satisfies conditions (i) and (ii) of Section~\ref{section:GH} and the
canonical slice is non-compact, then the short metric is incomplete.\qed
\end{thm}

From the definitions of $\omega_i$ we have $d\omega_i=d(e^t\alpha_i)|_{TS}$.
A straightforward computation yields the following structure
equations for the taut contact sphere $(\omega_1,\omega_2,\omega_3)$
on the canonical slice:
\begin{equation}
\label{eqn:structure-omega}
d\omega_i=\beta\wedge\omega_i+\omega_j\wedge\omega_k.
\end{equation}
These structure equations can be used to give an alternative definition
of the $1$-form~$\beta$.

The next proposition is the last ingredient we need in order to display the
announced Carnot--Carath\'eodory phenomenon.

\begin{prop}
\label{prop:CC}
Given a Riemann surface $(\Sigma ,J)$ and holomorphic
data $(\Phi ,\varphi )$, consider the corresponding
hyperk\"ahler structure on $W=\Sigma\times\R_t\times S^1_\theta$, as
explained in Section~\ref{section:GH}. Let
$(\omega_1,\omega_2,\omega_3)$ be the taut contact sphere induced on
the canonical slice~$S$, and let $\beta$ be the $1$-form on $S$
defined by the structure equations~(\ref{eqn:structure-omega}). If $\varphi$
is non-constant, then $\beta$ vanishes only along a discrete set of
orbits of the $S^1$-action on $S$, and it is a contact form in the rest
of~$S$, defining the opposite orientation to that defined
by the contact forms~$\omega_i$.
Therefore, any two points on $S$ can be joined by a path
tangent to $\ker\beta$.
\end{prop}

The proof of this proposition is given in the appendix.  We can use
the paths tangent to $\ker\beta$ to define Carnot--Carath\'eodory
distances on $S$. The one induced by $g_3$ is a complete distance,
because $g_3$ is a complete metric. Now the relation
$g_3=g_s+\beta^2$ tells us that the {\em incomplete} metric $g_s$
coincides with $g_3$ on those paths and thus induces exactly the
same Carnot--Carath\'eodory distance as~$g_3$. So we have an incomplete
metric $g_s$ inducing a complete Carnot--Carath\'eodory distance. The
geometric interpretation of this fact is that the proper paths
tangent to $\ker\beta$ are so wrinkled that they always have
infinite length in the incomplete metric~$g_s$. This phenomenon is
interesting because such paths can be arbitrarily $C^0$-close to any
given path --- so $g_s$ is, in some sense, very close to being
complete.
\section*{Appendix}
Here we prove Proposition~\ref{prop:CC}. We first derive an
explicit formula for $\beta$, from which all the properties
claimed in Proposition~\ref{prop:CC} can then be deduced.

In analogy to Section~\ref{section:GH}, we define the momentum map
$x=(x_1,x_2,x_3)\co W\to\R^3\,$ by
\[ x_i=\Omega_i (\partial_\theta
,\partial_t)=-e^t\alpha_i(\partial_\theta ) ,\]
and we define $\rho\co S\to\R$ by $\rho =|x|$.

Observe that $x$ is $\partial_{\theta}$-invariant, which allowed us
in Section~\ref{section:GH} to view it as a function
on $\Sigma\times\R_t$. Below, however, we want to consider
the $x_i$ as functions on~$S$, where they equal
$-\omega_i(\partial_{\theta})$. On $TS$ we then have
the identity $dx_i=\partial_{\theta}\ip d\omega_i$.

We also regard the functions $\varphi$
and $\psi =-1/\varphi$ as functions on the canonical slice~$S$, by
first pulling them back to $\Sigma\times\R_t\times S^1_\theta$  and
then restricting them to~$S$.

\begin{lem}
\label{lem:psi}
On $S$ we have the identity $\psi =-\beta (\partial_\theta )+\bi\rho$.
\end{lem}

\begin{proof}
Lemma~\ref{lem:slice} states that on $S$ the imaginary part
of $\psi$ is~$\rho$. In order to determine the real part,
we use the following
alternative expressions for the hyperk\"ahler metric $g$ on~$W$:
\begin{equation}
\label{eqn:g}
V^{-1}\cdot (d\theta +\eta )^2+V\cdot (dx_1^2+dx_2^2+dx_3^2) =
e^t\bigl( \Lambda_0^{-1}\cdot (dt+\beta_0)^2+\Lambda_0\cdot
(\alpha_1^2+\alpha_2^2+\alpha_3^2)\bigr) .
\end{equation}
Computing $g(\partial_\theta ,\partial_t)$ in the two possible ways,
we get $V^{-1}\eta (\partial_t)=e^t\Lambda_0^{-1}\beta_0
(\partial_\theta )$. On the canonical slice we have $e^t=\Lambda_0$
and, by Lemma~\ref{lem:slice}, $V=|\varphi |^2$. Moreover,
$\eta (\partial_t)=\mathrm{Re}\,\varphi$ by the definition
of $\varphi$ in Theorem~\ref{thm:GH}. So on $S$ we have
\[ \beta (\partial_\theta )=\beta_0(\partial_{\theta})=
\frac{\mathrm{Re}\,\varphi}{|\varphi |^2}=-\mathrm{Re}\,\psi ,\]
as claimed.
\end{proof}

From (\ref{eqn:g}) we find the following alternative expressions
for the metric $g_3$ induced on~$S$:
\begin{equation}
\label{eqn:g3}
|\varphi|^{-2}\cdot (d\theta +\eta )^2+|\varphi |^2
\cdot (dx_1^2+dx_2^2+dx_3^2) =
\beta^2+\omega_1^2+\omega_2^2+\omega_3^2.
\end{equation}
Introduce the auxiliary $1$-form $\gamma :=\sum_{i=1}^3x_i\omega_i$ on~$S$.
By taking the interior product with
$\partial_\theta$ in (\ref{eqn:g3}) we find, with Lemma~\ref{lem:psi},
\begin{equation}
\label{eqn:1}
|\varphi |^{-2}\cdot (d\theta +\eta )=
\beta (\partial_\theta )\beta -\gamma =
-(\mathrm{Re}\,\psi ) \beta -\gamma.
\end{equation}
With the structure equations~(\ref{eqn:structure-omega}) for
$d\omega_i$, we obtain from $dx_i=\partial_\theta\ip d\omega_i$
the equations
\[ dx_i=\beta (\partial_\theta )\omega_i+x_i\beta
-x_j\omega_k+x_k\omega_j .\]
This yields
\begin{equation}
\label{eqn:2}
\rho\, d\rho =\sum_ {i=1}^3x_i\, dx_i=
\rho^2\beta +\beta (\partial_\theta)\gamma =
\rho^2\beta -(\mathrm{Re}\,\psi )\gamma .
\end{equation}
Formul\ae\ (\ref{eqn:1}) and~(\ref{eqn:2}) constitute a linear system
for $\beta$ and~$\gamma$.
We solve it for $\beta$, observing that from Lemma~\ref{lem:psi}
we have $\beta (\partial_{\theta})^2+\rho^2=|\psi |^2$, to obtain
\begin{eqnarray*}
\beta & = & |\psi |^{-2}\cdot\bigl(
            \rho\, d\rho -(\mathrm{Re}\,\psi )|\varphi |^{-2}\cdot
            (d\theta +\eta )\bigr) \\
      & = & |\psi |^{-2}\rho\, d\rho - (\mathrm{Re}\,\psi )
            (d\theta +\eta ).
\end{eqnarray*}
Using the expression $\eta =(\mathrm{Re}\,\varphi)\frac{d\rho}{\rho}+\xi
=-|\psi |^{-2}(\mathrm{Re}\,\psi )\frac{d\rho}{\rho}+\xi$
from Theorem~\ref{thm:GH}, as well as Lemma~\ref{lem:psi} and the identities
in Lemma~\ref{lem:slice}, we transform the last equality as follows:
\begin{eqnarray*}
\beta & = & |\psi |^{-2}\cdot \bigl(\rho^2 +(\mathrm{Re}\,\psi )^2\bigr)
            \frac{d\rho}{\rho}
            -(\mathrm{Re}\,\psi )(d\theta +\xi )\\
      & = & \frac{d\rho}{\rho} - (\mathrm{Re}\,\psi )(d\theta +\xi )\\
      & = & d\log\mathrm{Im}\,\psi  - (\mathrm{Re}\,\psi )(d\theta +\xi ).
\end{eqnarray*}
This is the desired explicit expression for~$\beta$.

We next want to determine the subset of $S$ where $\beta$ vanishes.
Write the canonical slice as $S=G\times S^1_{\theta}$ as
in Lemma~\ref{lem:slice}, and points in $G$ as $(p,t(p))$ with
$p\in\Sigma$. We then see that $\beta$
vanishes precisely along the circles $\{ (p,t(p))\}\times S^1_\theta$
for those points $p\in\Sigma$ where $d\,\mathrm{Im}\,\psi$ and
$\mathrm{Re}\,\psi$ vanish simultaneously. It is easy to verify
that for  a non-constant holomorphic function $\psi$ such points form a
discrete subset of~$\Sigma$. Therefore $\beta$ is non-zero outside a
discrete set of circles, all of which are orbits of the $S^1$-action on the
canonical slice.

Finally, we want to prove that $\beta$ is a contact form outside the
described vanishing set. From $d(d\omega_i)=0$ and the structure
equations~(\ref{eqn:structure-omega}) one gets
\[ d\beta\wedge\omega_i+\beta\wedge\omega_j\wedge\omega_k=0.\]
Write $\beta =b_1\omega_1+b_2\omega_2+b_3\omega_3$. Then
\[ d\beta = - b_1\,\omega_2\wedge\omega_3-b_2\,\omega_3\wedge\omega_1-b_3
\,\omega_1\wedge\omega_2,\]
and so
\[ \beta\wedge d\beta
=-(b_1^2+b_2^2+b_3^2)\,\omega_1\wedge\omega_2\wedge\omega_3
=-(b_1^2+b_2^2+b_3^2)\,\omega_i\wedge d\omega_i.\]
So $\beta$ is indeed a contact form where it is non-zero, and
we also observe that the induced orientation is opposite
to the one defined by the~$\omega_i$.

\begin{ack}
Dragan Vukoti\'c helped us with many
questions and discussions about analytic functions on the unit disc,
setting us on the right track several times. Francisco Mart\'{\i}n
Serrano and Antonio Mart\'{\i}nez L\'opez gave us valuable
information on the same subject. Maria Jos\'e Mart\'{\i}n G\'omez
provided us with the fundamental reference \cite{bcp83}. Ernesto
Girondo made the computer generated hyperbolic tessellation. It is a
pleasure to thank all of them.
\end{ack}

\end{document}